\let\ORIlabel\label
\let\ORIrefstepcounter\refstepcounter
\AddToHook{package/hyperref/before}{
	\let\label\ORIlabel 
	\let\refstepcounter\ORIrefstepcounter}
\documentclass[final,onefignum,onetabnum]{siamonline220329}

\usepackage[english]{babel}

\usepackage{booktabs}
\usepackage[letterpaper,top=2cm,bottom=2cm,left=3cm,right=3cm,marginparwidth=1.75cm]{geometry}

\newlength\figureheight
\newlength\figurewidth 
\usepackage{amsmath}
\usepackage{mymacros}
\usepackage{graphicx}
\usepackage{amsfonts}
\usepackage{lmodern}
\usepackage[T1]{fontenc}
\usepackage{graphicx}
\usepackage{epstopdf}
\usepackage{amssymb}
\usepackage{todonotes}
\usepackage[T1]{fontenc}
\usepackage{tikz,caption}
\usepackage[caption=false,font=footnotesize]{subfig}
\usepackage{pgfplots}
\pgfplotsset{compat=newest}
\pgfplotsset{plot coordinates/math parser=false}
\usepackage{url}
\usepackage{color}
\usetikzlibrary{plotmarks} 
\usetikzlibrary{external}
\usepgfplotslibrary{external} 
\tikzset{external/mode=graphics if exists}
\usepackage{algpseudocode}

\newcommand{\eps}{\varepsilon}
\newcommand{\epsgl}{\mu}
\def\abs#1{\left|#1\right|} 
\def\norm#1{\left\|#1\right\|} 

\newcommand{\n}[1]{\| #1 \|}

\title{Fast and Simple Multiclass Data Segmentation: An Eigendecomposition and  Projection-Free Approach\thanks{Submitted to the editors August 13, 2025.
\funding{ 
The research of M.P.  was  partially granted by the Italian Ministry of University and Research (MUR) through the PRIN 2022 ``MOLE: Manifold constrained Optimization and LEarning'',  code: 2022ZK5ME7 MUR D.D. financing decree n. 20428 of Nov. 6th, 2024 (CUP B53C24006410006), and by PNRR - Missione 4 Istruzione e Ricerca - Componente C2 Investimento 1.1, Fondo per il Programma Nazionale di Ricerca e Progetti di Rilevante Interesse Nazionale (PRIN) funded by the European Commission under the NextGeneration EU programme, project ``Advanced optimization METhods for automated central veIn Sign detection in multiple sclerosis from magneTic resonAnce imaging (AMETISTA)'',  code: P2022J9SNP,
MUR D.D. financing decree n. 1379 of 1st Sept. 2023 (CUP E53D23017980001). Her work was also partially supported by the INdAM-GNCS project CUP\_E53C24001950001.
The research of F.R. has been partially funded by the EuropeanUnion - NextGenerationEU under the National Recovery and Resilience Plan (NRRP), Mission 4 Component 2 Investment
1.1 - Call PRIN 2022 No. 104 of February 2, 2022 of Italian Ministry of University and Research;
Project 2022BMBW2A (subject area: PE - Physical Sciences and Engineering)``Large-scale optimization
for sustainable and resilient energy systems.''
}
}
}

\author{Chiara Faccio\thanks{Department of Mathematics "Tullio Levi-Civita"
University of Padova, Italy. C.F. is member of the INdAM Research Group GNCS (\email{rinaldi@math.unipd.it, chiara.faccio@unipd.it}) }
\and Margherita Porcelli\thanks{Dipartimento di Ingegneria Industriale (DIEF)
Università degli Studi di Firenze, Italy,  ISTI--CNR, Italy and member of the INdAM Research Group GNCS (\email{margherita.porcelli@unifi.it})  }
\and  Francesco Rinaldi\footnotemark[2] 
\and Martin Stoll\thanks{Department of Mathematics, Chemnitz University of Technology, Germany (\email{martin.stoll@math.tu-chemnitz.de})}}

\ifpdf
\hypersetup{ pdftitle={A Greedy Frank Wolfe For Multiclass Data Segmentation} }
\fi

\begin{document}
\maketitle
\begin{abstract}
Graph-based machine learning has seen an increased interest over the last decade with many connections to other fields of applied mathematics. Learning based on partial differential equations, such as the phase-field Allen-Cahn equation, allows efficient handling of semi-supervised learning approaches on graphs. The numerical solution of the graph Allen-Cahn equation via a convexity splitting or the Merriman-Bence-Osher (MBO) scheme, albeit being a widely used approach, requires the calculation of a graph Laplacian eigendecomposition and repeated projections over the unit simplex to maintain valid partitions. The computational efficiency of those methods is hence limited by those two bottlenecks in practice, especially when dealing with  large-scale instances.\\
In order to overcome these limitations, we propose a new framework combining a novel penalty-based reformulation of the segmentation problem, which ensures valid partitions (i.e., binary solutions) for appropriate parameter choices, with an eigendecomposition and projection-free optimization scheme, which has a small per-iteration complexity (by relying primarily on sparse matrix-vector products) and guarantees good convergence properties.\\ 
Experiments on synthetic and real-world datasets related to data segmentation in networks and images demonstrate that the proposed framework achieves comparable or better accuracy than the CS and MBO methods while being significantly faster, particularly for large-scale problems. 
\end{abstract}
\section{Motivation}
Many problems can be formulated as learning problems defined on a graph-domain (see, e.g., \cite{chami2022machine} for a detailed discussion on this matter). 
In this paper, a graph $G=(V,E)$ consists of a set of vertices $V$ with $\abs{V}=n$ and an edge set $E$ as a subset of $V\times V.$ The graphs we consider are undirected and consist of one component. In the case of multiple components, the methods proposed here can be used for the individual subgraphs. The focus of our work is a semi-supervised learning framework where a subset of the nodes are assigned to a label and one wants to infer the label of the remaining nodes. While in binary classification one assumes to have only two labels, the focus of this work is on assigning each node to one of $K>2$ labels.

A key ingredient in graph-based learning is the representation of the graph information in terms of linear algebra components. For this, one typically considers the weighted adjacency matrix $\bW\in\R^{n,n}$ where $w_{i,j}\geq 0$ describes the strength of the connection between the different nodes. The weight between two nodes is often modeled using the Gaussian similarity measure
$$
w_{ij}=\mathrm{exp}\left(-\frac{\norm{x_i-x_j}^2}{2\sigma^2}\right),
$$
where $\sigma$ is a chosen hyperparameter, $x_i$ and $x_j$ are two $d$ dimensional feature vectors associated with the nodes $v_i$ and $v_j,$ respectively.  As a measure of importance of the nodes in the network one typically computes the degree $d_i$ for node $v_i$ using 
$
d_i=\sum_{j=1}^{n}w_{ij},
$
which corresponds to the row-sum of the weight matrix $\bW$. This leads to the definition of the \textit{graph Laplacian} 
$$
\bL=\bD-\bW,
$$
where $\bD$ is a diagonal matrix collecting the degrees of all nodes. The graph Laplacian has been used extensively and we refer to \cite{Chu97, VLu07} for excellent reviews on the graph Laplacian and its applications. One often considers versions of the Laplacian, namely, the symmetric normalized Laplacian $\bL_s=\bI-\bD^{-1/2}\bW\bD^{-1/2}$ or the random walk Laplacian 
$\bL_w=\bI-\bD^{-1}\bW,$ where the assumption made on the graph guarantees the invertibility of the matrix $\bD.$ 

In semi-supervised learning, the node labels are given for a potentially small subset of the nodes and we want to learn the labels of the remaining nodes while simultaneously leveraging labeled and unlabeled nodes. To solve this problem, approaches based on partial differential equations (PDEs) have become popular (see, e.g., \cite{BerF12,GarMBFP14}) and while these often show very good performance, they typically are implemented relying on the a priori computation of a suitably chosen subset of the eigenvalues and eigenvectors of the graph Laplacian to reduce the complexity during the solution process. The computation of eigeninformation typically requires the use of a Krylov subspace, which in turns requires matrix vector products with the graph Laplacian. Low-rank approaches such as the Nystr\"om approximation are often also based on matrix vector products with the graph Laplacian, especially when modern randomized linear algebra methods \cite{martinsson2020randomized} are employed. This means that the truncated eigendecomposition needed for an efficient implementation of the PDE-based schemes like, e.g., the convexity splitting (CS) \cite{van2014mean} and Merriman-Bence-Osher (MBO) \cite{MerKB13,van2014mean} approach considered in this paper, have non-negligible startup cost. This cannot be avoided as otherwise the mentioned schemes would require the solution of linear systems involving shifted versions of the large-scale graph Laplacian.  While this decomposition can obviously be precomputed, the final accuracy and computational efficiency of the methods strongly depend on the number of retained eigenpairs, an unknown hyperparameter that is very expensive to tune and, as we will see, varies depending on the problem structure. In this work, both the diffuse interface and the MBO scheme are  considered as reference methods among PDE-based approaches as  these have previously been compared favorably to other graph-based learning schemes in \cite{GarMBFP14}. For certain imaging classification task these methods have been compared to neural network based approaches \cite{bergermann2021semi} producing similar accuracy while only requiring a fraction of the computational cost.

Furthermore, these PDE-based schemes require at each iteration a tailored projection step to guarantee that the related iterate represents a valid partition. As we will see later on, such a step has a  ${\cal O}(nK \log K)$ cost for a problem with $n$ data points and $K$ classes.  While this may be still doable for small-scale instances, it becomes prohibitively slow when dealing with large-scale ones, like, e.g., social networks or high-resolution images, where $n$ and $K$ can easily reach millions and thousands, respectively. We hence need a way to get rid of these two bottlenecks in practice.

\subsection*{Contribution}

The contribution of this work is threefold. First, leveraging on the graph-based semi-supervised learning model proposed in \cite{BerF12}, we define a new and easier to handle penalty-based continuous formulation of the multiclass data segmentation problem. We analyze its theoretical properties and prove 
that the model obtains valid partitions (i.e., binary solutions)
 for a sufficiently small value of the penalty parameter.

Second, we develop a new algorithm in the class of Frank-Wolfe methods \cite{bomze2021frank}, named Greedy Frank-Wolfe (GFW), which suitably exploits the structure of the proposed model to keep the per-iteration cost  small enough (by relying only on tailored sparse matrix-vector products) and allows us to  avoid both the projection steps and eigendecomposition computations, which represent, as we highlighted above, the main bottlenecks for the CS and MBO approaches. We analyze the convergence rate of the method and discuss the conditions under which the method achieves one-shot convergence (i.e., convergence in one iteration). Given appropriate assumptions, we also provide a detailed complexity analysis of the proposed algorithm. Such a result, to the best of our knowledge, is not available for the PDE-based schemes considered in the paper.

Finally, we show the usefulness of the proposed framework. We provide experimental results on both benchmark data sets and real data (real networks and image labeling), demonstrating the effectiveness of our method compared to CS and MBO, both in terms of computational time and quality of the segmentation.

\subsection*{Organization of the paper}

The remainder of this paper is organized as follows. In Section~\ref{sec:model}, we introduce the mathematical models for multiclass data segmentation on graphs. Starting from classical diffuse interface models based on the Ginzburg–Landau energy, we derive the novel penalty-based continuous reformulation of the problem and analyze its theoretical properties. In Section \ref{sec:pdeschemes}, we briefly review two established PDE-inspired methods widely used for multiclass data segmentation, i.e., the convexity splitting (CS) and Merriman-Bence-Osher (MBO) schemes and discuss their computational bottlenecks.
Section \ref{sec:projectionfreeapproaches} presents the greedy variant of the Frank-Wolfe algorithm, suitably designed to exploit the structure of our new model, discusses the convergence properties of the method, and analyzes its computational complexity.
In Section \ref{sec:numexp}, we report numerical experiments that compare our method with CS and MBO on a variety of synthetic and real datasets, including network data and image segmentation tasks. Finally, Section \ref{sec:conclusions} presents our conclusions and discusses possible directions for future research.

\section{Models for multiclass segmentation}\label{sec:model}
Learning methods based on partial differential equation (PDE) concepts have gained more traction in the recent years, while machine learning methods for PDEs have also become a major topic in computational science and engineering \cite{cuomo2022scientific}. The objective function for the semi-supervised learning method we propose is derived from diffuse interface methods. Diffuse interface methods are a classical tool in the simulation of materials science problems \cite{allen1979microscopic,CahH58}. The derivation of classical models such as the Allen--Cahn \cite{allen1979microscopic} or Cahn--Hilliard equations \cite{CahH58} is typically obtained from a gradient flow of the Ginzburg--Landau energy
\begin{equation}
\label{GL_scalar}
E_{GL}(u)=\int_{\Omega}\frac{\epsgl}{2}\abs{\nabla u}^2dx+\int_{\Omega}\frac{1}{\epsgl}\psi(u)dx
\end{equation}
where $u$ is the phase-field function and $\epsgl$ the interface parameter, which is typically assumed to be small and related to the interface width describing the transition from one pure phase to the other within the domain $\Omega.$ 
The function $\psi(u)$ is a potential that forces the phase-field $u$ to take values at either $u\approx -1$ or $u\approx 1$ 
(assuming two pure phases). We come back to the discussion of the choice of potential as this is one of the contributions of this paper. The minimization of the energy $E_{GL}(u)$ follows a gradient flow, i.e., 
$$
\partial_{t}{u}=-\mathrm{grad}(E_{GL}(u)).
$$
Bertozzi and Flenner \cite{BerF12} formulated the semi-supervised learning problem based on graph quantities.
The energy function to minimize is now defined using the description of the underlying graph Laplacian $\bL_s$ to give
\begin{equation}
\label{gl1}
E_{GL}(\bu)=\frac{\epsgl}{2} \bu^{\top} \bL_s\bu+\sum_{x\in V}\frac{1}{\epsgl}\psi(\bu(x))+F_{\omega}(f,\bu),
\end{equation}
with the energy contribution $F_{\omega}(f,\bu)$ describing the fidelity term that would lead to $\omega(x)(f-\bu)$ (cf. \cite{BerF12}) in the continuous Allen--Cahn equation often used for image inpainting applications \cite{bertozzi2007analysis,bosch2015fractional}.  In practice, often more than two components are of interest. The continuous Ginzburg--Landau energy for two components in (\ref{GL_scalar}) generalizes to
\begin{equation}
\label{GL_vector}
E_{GL}({u})=\int_{\Omega}\frac{\epsgl}{2}\sum_{i=1}^{K}{\abs{\nabla u_{i}}^2}dx+\int_{\Omega}\frac{1}{\epsgl}\psi({u})dx,
\end{equation}
for $K>2$ components. Here, ${u}=(u_{1},\ldots,u_{K})^{\top}$ is now the vector-valued phase-field, and the potential function $\psi({u})$ has $K$ distinct minima instead of two. Note that with a slight abuse of notation use the subscript \textit{GL} for all objective functions based on the Ginzburg-Landau energy. 

Garcia-Cardona et al.~\cite{GarMBFP14} as well as Merkurjev et al.~\cite{MerGBFP14} have extended these continuous models to the graph domain. In the following, we briefly summarize their approach. We introduce the matrix 
$\bU=({\bu}_{1},\ldots,{\bu}_{n})^{\top}\in\mathbb{R}^{n\times K}$. Here, the $k$th component of
${\bu}_{i}\in\mathbb{R}^{K}$ is the strength for data point $i$ to belong to class $k$. 
For each node $i$, the vector ${\bu}_{i}$ has to be an element of the unit simplex:
$$\Delta^i_{K-1}:=\{\bu_i \in \R^K_+ :\mathbf{1}^\top \bu_i=1\}=\conv\{{\bf e}_j : j\in \irg 1K \},$$
so we have that $\bU$ belongs to the set:
\begin{equation}\label{fs}
\Sigma = \{\bU \in \R^{n\times K}:\bu_i \in \Delta^i_{K-1}, \ i=1, \dots, n\}. 
\end{equation}
In the text $\mathbf{1}$ is either a vector or matrix of all ones of appropriate dimensionality. The Ginzburg--Landau energy functional on graphs in (\ref{gl1}) generalizes to the multi-class case as
\begin{equation}
\label{gl_vv}
E_{GL}(\bU)=\frac{\epsgl}{2} \langle \bU, \bL_{s}\bU\rangle+\frac{1}{\epsgl}\psi(\bU)+F_{\omega}(\hat \bU,\bU),
\end{equation}
where the graph-cut term from the binary case is found as
\begin{displaymath}
 \langle \bU, \bL_{s} \bU\rangle=\textup{trace}(\bU^{\top}\bL_{s}\bU)
\end{displaymath}
and measures variations in the vector field. The potential term is given via 
\begin{equation}
 \label{eq:vv:smooth_pot}
 \psi(\bU)=\frac{1}{2}\sum_{i\in V}{\left(\prod_{k=1}^{K}{\frac{1}{4}||{\bu}_{i}-{\bf e}_{k}}||_{L_{1}}^{2}\right)}
\end{equation} 
drives the system closer to the pure phases as the minimum is achieved if ${\bu}_i$ is close to a unit vector, i.e., a pure phase. Finally, the fidelity term incorporating training data is given by
\begin{displaymath}
 F_{\omega}(\hat \bU,\bU)=\sum_{i\in V}{\frac{\omega_i}{2}||\hat{\bu}_{i}-{\bu}_{i}||^{2}_{L_{2}}} =
 \frac{1}{2}
 \| \bD_{\omega}^{1/2} (\hat \bU-  \bU)\|_F^2
\end{displaymath}
as it enables the encoding of a priori information with $\hat \bU=(\hat{\bu}_{1},\ldots,\hat{\bu}_{n})^{\top}.$ Indeed, $\omega_i$ is a parameter that
takes the value of a positive constant $\omega_0$ if $i$ is a fidelity node
and zero otherwise, and ${\hat \bu}_i$ is the known value of the fidelity node $i$. 
Here the matrix $\bD_{\omega}$ is a diagonal matrix with elements $\omega_i$ on the diagonal.

Note that the authors of~\cite{GarMBFP14,MerGBFP14} use an $L_{1}$-norm for the potential term as it prevents an undesirable local minimum  from occurring at the center of the 
simplex, as would be the case with an $L_{2}$-norm for large $K$. One standard way to minimize the above given energies is to follow a gradient flow approach and for the graph case, one can also employ a gradient scheme as was done in \cite{BerF12,bosch2018generalizing,MerGBFP14}. 

\subsection{A new penalty-based model for multiclass data segmentation}
In here, we propose a different formulation where the term $\psi(\bU)$ enforcing the pure phases is replaced with a new penalty term, that is 
\begin{displaymath}
 \phi(\bU)=\langle \bU^\top, \mathbf{1}-\bU^\top\rangle=\sum_{i\in V} \bu_i^\top (\mathbf{1}-\bu_i) = \textup{trace}(\bU(\mathbf{1}-\bU^\top)).
\end{displaymath}
 The Ginzburg--Landau energy functional on graphs in (\ref{gl1}) with this new penalty term for the multi-class case can be written as follows
\begin{equation}
\label{gl_vv2}
E(\bU)=\frac{1}{2} \langle \bU, \bL_{s}\bU\rangle+\frac{1}{\varepsilon}\phi(\bU)+F_{\omega}(\hat \bU,\bU),
\end{equation}
and the problem we have is
\begin{equation}\label{penalty_form}
\argmin_{\bU\in \Sigma} E(\bU).
\end{equation}
Both parameters $\epsgl$ and $\eps$ play a role in enforcing the pure phases, in the objectives (\ref{gl_vv}) and (\ref{gl_vv2}),  respectively. The parameter $\epsgl$ originates from materials science modeling and corresponds to a smooth transition from one phase to the other, the parameter $\eps$ is purely motivated as a penalty parameter in the objective function.
We now calculate the diameter of the feasible set $\Sigma$, a technical result that will be useful for the analyses that follow. To improve readability and enhance the  presentation flow, all proofs are reported in the Appendix.

\begin{proposition}\label{pr:diam} The diameter of the set $\Sigma$ is $\sqrt{2n}$.
\end{proposition}

The new energy function in (\ref{gl_vv2}) takes the form
\begin{equation}
\label{gradgl_vv2}
E(\bU)= \frac{1}{2} \textup{trace}(\bU^\top \bL_{s}\bU)+\frac{1}{\varepsilon}\textup{trace}(\bU (\mathbf{1} - \bU^\top))+ \frac{1}{2}\|\bD_\omega^{1/2} (\hat \bU - \bU)\|_F^2,
\end{equation}
and we then obtain its gradient
\begin{equation}
\label{gradgl_vv2g}
\nabla E(\bU)= \bL_{s}\bU+\frac{1}{\varepsilon}(\mathbf{1} - 2 \bU) - \bD_\omega (\hat \bU - \bU).
\end{equation}
It is important to observe that the objective has Lipschitz continuous gradient. In the next proposition, we report a bound on the Lipschitz constant of the gradient. 
\begin{proposition}\label{pr:LCG}
The gradient of $E(\bU)$ is Lipschitz continuous with constant 
\begin{equation}\label{eq:LCGineq}
L\leq \bar L= \lambda_{max}(\bL_s)+\frac{2}{\varepsilon} +\,d^{max}_\omega,
\end{equation}
with $\lambda_{max}(\bL_s)$ maximum eigenvalue of the matrix $\bL_s$ and $d^{max}_\omega=\max_{i}(\bD_{\omega})_{ii}=\omega_0$.
\end{proposition}

It is important to highlight that when $\varepsilon$ is sufficiently small we can guarantee that  the solution matrix $\bU$ is binary, i.e., problem \eqref{penalty_form} is equivalent to 
\begin{equation}\label{binary_form}
\argmin_{U\in \bar \Sigma} \bar E(\bU),
\end{equation}
with $$\bar E(\bU)=\frac{1}{2} \langle \bU, \bL_{s}\bU\rangle+F_{\omega}(\hat \bU,\bU)$$
and 
$\bar\Sigma=\{ \bU=({\bu}_{1},\ldots,{\bu}_{n})^{\top}\in\mathbb{R}^{n, K}: \mathbf{1}^\top \bu_i=1, \ \bu_i\in \{0,1\}^K,\  \forall\  i\in[1:n]\}.$
This equivalence result is stated below.
\begin{theorem}\label{th:equivalence}
Let  $\bar \varepsilon<2 / \lambda_{max}(\bL_{s}+\bD_{\omega})$. 
Then, for all $\varepsilon\in (0, \bar\varepsilon]$, we have:

\begin{itemize}
\item [(i)] Every local/global minimizer $\bU^*$ of Problem \eqref{penalty_form} is strict and such that $\bU^*\in\{0,1\}^{n\times K}$. 
\item [(ii)]Problems \eqref{penalty_form} and \eqref{binary_form} have the same global minimizers, that is
$$\argmin_{U\in  \Sigma} E(\bU)=\argmin_{U\in \bar \Sigma} \bar E(\bU)\,.$$
\end{itemize} 
\end{theorem}

The above result therefore guarantees that even when using a local solver to handle the problem \eqref{penalty_form}, we are able to obtain a solution that has binary components in the end. 
    Although our formulation is primarily designed to produce hard-membership solutions, selecting a moderately large value for 
the parameter $\varepsilon$ naturally leads to fractional values. In the same way as with softmax-based scores,
these values can therefore be  seen as  soft memberships or pseudo-confidences. It is also important to note that  the upper bound $2 / \lambda_{max}(\bL_{s}+\bD_{\omega})$ on the parameter $\varepsilon$ just enforces a theoretical guarantee rather than representing a requirement for practical use.
In all of our experiments, the method indeed showed robustness with respect to a broad range of parameter values, and a simple inexpensive tuning procedure was sufficient even on large graphs.

Before introducing our Frank-Wolfe scheme, we briefly recall two of the most popular graph-based methods for multiclass graph segmentation.

\section{Existing PDE-based approaches: convexity splitting and MBO scheme}\label{sec:pdeschemes}

A possible way to address the multi-class problem is by means of a  convexity splitting method \cite{bertozzi2006inpainting,bosch2018generalizing,GarMBFP14}, which involves decomposing  the energy function $E_{GL}$ in (\ref{gl_vv}) into convex and concave parts: 
\begin{displaymath}
 E_{GL}(\bU)=E_1(\bU)-E_2(\bU)
\end{displaymath}
with
\begin{displaymath}
 E_1(\bU)=\frac{\epsgl}{2}\langle \bU, \bL_{s}\bU\rangle+\frac{c}{2}\langle \bU,\bU\rangle,\quad 
E_2(\bU)=-\frac{1}{\epsgl}\psi(\bU)-F(\hat \bU,\bU)+\frac{c}{2}\langle \bU,\bU\rangle,
\end{displaymath}
with $c\in \mathbb{R}$ denoting a constant that is chosen to guarantee the convexity/concavity of the energy terms. Indeed, we require $c\geq\omega_{0}+\frac{1}{\epsgl}$; 
see \cite[p.~6]{GarMBFP14}. The convexity splitting scheme results in
\begin{equation}
 \label{CS_vv}
\frac{\bU^{\ell+1}-{\bU}^{\ell}}{\tau}+\epsgl \bL_{s}\bU^{\ell+1}+c\bU^{\ell+1}=-\frac{1}{2\epsgl}T({\bU^{\ell}})+c{\bU^{\ell}}+\bD_{\omega}(\hat{\bU}-{\bU^{\ell}}),
\end{equation}
where the elements $T_{ij}$ of the matrix $T(\bU^{\ell})$ are given as
\begin{equation}
     T_{ij}=\sum_{l=1}^{K}{\frac{1}{2}\left(1-2\delta_{jl}\right)||{\bu}_{i}^{\ell}-{\bf e}_{l}||_{L_{1}}}\prod_{m=1,m\neq l}^{K}{\frac{1}{4}||{\bu}_{i}^{\ell}-{\bf e}_{m}}||_{L_{1}}^{2}
\end{equation}\label{eq_Tik}
and $\tau>0$ is the time step.

Here, $\ell$ indicates the old time-step and $\ell+1$ the new time-step. Using the eigendecomposition of the graph Laplacian as
$\bL_s=\Phi\Lambda\Phi^{\top}$,  $\bU$ can be expressed in this basis via $\bU=\Phi_k\cU$
with $\Phi_k\in\R^{n,k}$ collecting the dominating eigenvalues of the graph Laplacian and $\cU$ the coefficients for the expansion in those vectors. By plugging this into the above equation and multiplying on the left by $\Phi_k^{\top}$ to obtain the reduced equation
\begin{equation}
 \label{CS_vv_2}
 \mathcal{U}^{\ell+1}=B_k^{-1}\left[(1+c\tau){\mathcal{U}^{\ell}}-\frac{\tau}{2\epsgl}\Phi_k^{\top}T({\bU}^{\ell})+\tau \Phi_k^{\top}\bD_{\omega}(\hat{\bU}-{\bU}^{\ell})\right].
\end{equation}
where $B_k=(1+c\tau)I+\epsgl\tau\Lambda_k$ and
$\Lambda_k$ is the diagonal matrix of $k$ dominating eigenvalues. Since $B_k$ is a diagonal matrix with positive entries, its inverse is easy to apply. At every iteration, i.e., time-step, every $i$th row of the approximation $\mathcal{U}^{\ell+1}$ has to be projected onto the unit simplex $\Delta_{K-1}^i$ to avoid explosion of the solution coefficients (see, e.g., Algorithm \ref{alg:proj}).

Alternatively, one can use the Merriman-Bence-Osher (MBO) \cite{merriman1994motion} scheme, which was originally used for the binary classification in \cite{MerGBFP14}. It is based on a two-step procedure for the continuous model based on a \textit{diffusion step:} Obtain $u^{\ell+\frac{1}{2}}$  from the solution of 
    $$
    \frac{\partial u}{\partial t}=\Delta u,
    $$
    followed by a \textit{thresholding step:} Compute 
    $$
    u^{\ell+1} = 
        \left\{\begin{array}{ll}
            1 & : u^{\ell+\frac{1}{2}}\geq 0,\\
            -1 & :u^{\ell+\frac{1}{2}} <0.\\
        \end{array}\right.
    $$
This scheme was later extended to the case of graph-based semi-supervised learning in \cite{BerF12} for the binary classification case and for multiclass case in \cite{GarMBFP14}. There, the authors propose the following procedure. Perform a \textit{diffusion step:} Obtain $\bU^{\ell+\frac{1}{2}}$ from  
    $$
    \frac{\bU^{\ell+\frac{1}{2}}-\bU^{\ell}}{\tau}=-\bL_s \bU^{\ell+\frac{1}{2}}-\bD_{\omega}(\hat{\bU}-\bU^{\ell}),
    $$
and then perform a \textit{thresholding step:} Compute 
    $$
    \bu_i^{\ell+1} = e_k,
    $$
    where $e_k$ is the vertex that is closest to the projection of $\bu_i^{\ell+\frac{1}{2}}$ to the simplex. This problem can again be tackled by using a truncated eigendecomposition. We then obtain
\begin{equation}\label{mbo_update}
   \cU^{\ell+\frac{1}{2}}=\left(\bI+\tau \Lambda_k\right)^{-1}\left[\cU^{\ell}- \tau  \Phi_k^{\top}\bD_{\omega}(\hat{\cU}-\Phi_k\cU^{\ell})\right]. 
\end{equation}
Instead of the updating for one step we can also perform a number of such steps via
\begin{itemize}
    \item $\cU^{\ell+\frac{1}{2}}=\left(\bI+\tau \Lambda_k\right)^{-1}\left[\cU^{\ell}- \tau  \Phi_k^{\top}\bD_{\omega}(\hat{\cU}-\Phi_k\cU^{\ell})\right]$
    \item $\cU^{\ell+1}=\mathrm{projspx}(\cU^{\ell+\frac{1}{2}})$
    \item $\cU^{\ell}=\cU^{\ell+1},$
\end{itemize}
where projspx denotes the row-wise projection onto the unit simplex (see Algorithm \ref{alg:proj}).
Once the reduced eigendecomposition $(\Lambda_k,\Phi_k)$ is available, the computational cost per step of MBO is small as a very small linear system has to be solved.

In both the approaches, CS and  MBO, if the cost of the eigendecomposition is negligible, the computational bottleneck is the projection onto the unit simplex.

We conclude this section  with the most common procedure to project a vector $y\in \R^{K}$ onto the simplex $\Delta^i_{K-1}$, commonly used in the implementation of CS and MBO. The algorithm is the sorting-based method (see,e.g.,  \cite{condat2016fast} and \cite{held1974validation}  for further details) summarized in Algorithm \ref{alg:proj}.
\begin{algorithm}[H]
			\caption{\texttt{Projection onto the unit simplex }\cite[Algorithm 1 (\texttt{projsplx})]{chen2011projection}}\label{alg:proj}
            	\begin{algorithmic}
				\par\vspace*{0.1cm}
				\item$\,\,\,1$\hspace*{0.1truecm} Input $Y =(y_1,\ldots, y_K )^\top \in \R^{K}$.
               \item$\,\,\,2$\hspace*{0.1truecm}  Sort $y$ as $\hat y$ in the ascending order as $\hat y_1\leq \ldots \leq \hat y_K$ , and set $i = K-1$.
 \item$\,\,\,3$\hspace*{0.1truecm}  Compute 
    $$
    t_i=\frac{\left(\sum_{j=i+1}^{K}\hat y_j\right)-1}{K-i}.
    $$
 \,\,\,\hspace*{0.2truecm}   If $t_i\geq \hat y_{i}$ set $\hat t=t_i$ and go to Step 5.\\
 \,\,\,\hspace*{0.2truecm}   Else set $i:=i-1$ and go to Step 3 if $i\geq 1$, or go to Step 4 if $i=0.$
    \item$\,\,\,4$\hspace*{0.1truecm} Compute
    $$
    \hat t=\frac{\left(\sum_{j=i+1}^{\ell}\hat y_j\right)-1}{K-i}.
    $$
    \item$\,\,\,5$\hspace*{0.1truecm}  Return $x=(y-\hat t)_{+}$ as the projection of $y$ onto the simplex.
                \end{algorithmic}

\end{algorithm}
At step 5 we have $(y-\hat t)_{+}:=\max\left\lbrace 0,y-\hat t\right\rbrace$, where the maximum is understood in a component-wise fashion. The computational cost of this algorithm is $\cO(K\log K)$ for the $K$-class problem. Note that both the convexity splitting and the MBO method require then $\cO(nK\log K)$ for the projection step in every iteration. Other methods with an expected linear time complexity exist (see, e.g., \cite{condat2016fast}), but they are usually slower in practice.

\section{Projection-free methods: a greedy Frank-Wolfe algorithm}\label{sec:projectionfreeapproaches}
In this section, we focus on some first-order methods that have recently garnered significant attention in both the machine learning and optimization communities. These methods make use of a Linear Minimization Oracle (LMO) to minimize a sequence of linear functions over the original feasible set.
We will introduce the Frank-Wolfe (FW) approach 
(the interested reader is referred to, e.g.,~\cite{bomze2021frank} for further details on this method)  and a tailored variant that allow to handle the specific problem we are dealing with.

The classical FW method for minimization of a smooth objective $E$ generates a sequence of feasible points $\{\bU_k\}$ following the scheme of \Cref{alg:FW}. At the iteration $k$ it moves toward a vertex, i.e., an extreme point, of the feasible set minimizing the scalar product with the current gradient $\nabla E(\bU_k)$ \footnote{an $n\times K$ matrix whose element $(i,j)$ is $\frac{\partial E(X)}{\partial X_{i,j}}$ }. It therefore  makes use of a linear minimization oracle (LMO) for the feasible set $\Sigma$, that is, 
	\begin{equation}\label{eq:linsubp}
		\LMO_{\Sigma}(\bG) \in \argmin_{\bZ \in \Sigma} \ps{\bG}{\bZ} \, ,
	\end{equation} 
	defining the descent direction as
  	\begin{equation} \label{fwdk}
 		 \bD_k:= \bS_k - \bU_k, \ \ \bS_k \in \LMO_{\Sigma}(\nabla E(\bU_k)) \, .
 \end{equation}
 	In particular, the update at step 6 can be written as
 	\begin{equation}
 		\bU_{k + 1} = \bU_k + \alpha_k (\bS_k - \bU_k) = \alpha_k \bS_k + (1 - \alpha_k) \bU_k.
 	\end{equation}
 	Since $\alpha_k \in [0, 1]$, by induction $\bU_{k + 1}$ can be written as a convex combination of elements in the set ${\cal S}_{k + 1} := \{\bU_0\} \cup \{\bS_i\}_{0 \leq i \leq k}$. 
    When ${\cal C }= \conv({\cal A})$ for a set $\cal A$ of matrices with some common property, usually called "elementary atoms", if $\bU_0 \in \cal A$ then $\bU_{k}$ can be written as a convex combination of $k + 1$ elements in $\cal A$. In our case we have that $ \cal A$  is given by  all $n\times K$ matrices having as rows only elements of the canonical basis $e_j^\top$ with $j\in[1:K]$. 
	
\begin{algorithm}[H]
			\caption{ \texttt{Frank-Wolfe method }\cite{frank1956algorithm}}
			\label{alg:FW}
			\begin{algorithmic}
				\par\vspace*{0.1cm}
				\item$\,\,\,1$\hspace*{0.1truecm} Choose a point $\bU_0\in \Sigma$
				\item$\,\,\,2$\hspace*{0.1truecm} For $k=0,\ldots$
				\item$\,\,\,3$\hspace*{0.9truecm} Compute $\bS_k \in \LMO_{\Sigma}(\nabla E(\bU_k))$
    \item$\,\,\,4$\hspace*{0.9truecm} If $\bS_k$ satisfies some specific condition, then STOP 
				\item$\,\,\,5$\hspace*{0.9truecm} Set $\bD_k =  \bS_k - \bU_k$ \ \ 
				\item$\,\,\,6$\hspace*{0.9truecm} Set $\bU_{k+1}=\bU_k+\alpha_k \bD_k$, with $\alpha_k\in (0,1]$ a suitably chosen stepsize
				\item$\,\,\,7$\hspace*{0.1truecm} End for
				\par\vspace*{0.1cm}
			\end{algorithmic}
\end{algorithm}

 The \LMO\ for the feasible set \eqref{fs} is defined as follows: 
	$$\LMO_{\Sigma}(\bC)=({\bf e}_{j_1},\ldots,{\bf e}_{j_n})^{\top} ,$$  
	with $\bC=(\bc_1,\ldots, \bc_n)^\top$and  $j_i \in \argmin_{j\in [1:K]} (\bc_i)_j$. 
		It is easy to see that the \LMO\  cost   is~$\mathcal{O}(n\cdot K)$.

We now define  a greedy variant that  better exploits the structure of the problem at hand.

\begin{algorithm}[H]
			\caption{\texttt{Greedy Frank-Wolfe (GFW) method}}
			\label{alg:FW2}
			\begin{algorithmic}
				\par\vspace*{0.1cm}
				\item$\,\,\,1$\hspace*{0.1truecm} Choose a point $\bU_0\in \Sigma$
				\item$\,\,\,2$\hspace*{0.1truecm} For $k=0,\ldots$
				\item$\,\,\,3$\hspace*{0.9truecm} Compute $\bS_k \in \GLMO_{\Sigma}(\nabla E(\bU_k))$
    \item$\,\,\,4$\hspace*{0.9truecm} If $\bS_k$ satisfies some specific condition, then STOP 
				\item$\,\,\,5$\hspace*{0.9truecm} Set $\bD_k =  \bS_k - \bU_k$ \ \ 
				\item$\,\,\,6$\hspace*{0.9truecm} Set $\bU_{k+1}=\bU_k+\alpha_k \bD_k$, with $\alpha_k\in (0,1]$ a suitably chosen stepsize
				\item$\,\,\,7$\hspace*{0.1truecm} End for
				\par\vspace*{0.1cm}
			\end{algorithmic}
\end{algorithm}

As we can easily see, the structure of the algorithm stays the same apart from the \LMO, now replaced by the so-called \emph{Greedy Linear Minimization Oracle} (\GLMO), which significantly reduces the per-iteration computational cost, while preserving the convergence properties of the original method.  We can define the \GLMO\  as follows: 
$$\GLMO_{\Sigma}(\bG)=({\bf w}_{1},\ldots,{\bf w}_{n})^{\top} ,$$  
with $\bG=(\bg_1,\dots,\bg_n)^\top=\nabla E(\bU)$,
${\bf w}_{i}={\bf e}_{j_i}$, if $i\in {\cal I}=\{i\in [1:n]:\ {\bf u}_i \notin\{0,1\}^K\}$, $j_i \in \argmin_{j\in supp(\bu_i)} (\bg_i)_j$,
and  ${\bf  w}_{i}={\bu_i}$ otherwise. Here, $supp(\bu_i)$ denotes the indices of the nonzero elements of the vector $\bu_i$. More specifically, in the \GLMO\ we solve for each $i$ such that ${\bf u}_i \notin\{0,1\}^K$ the following problem: 
    \begin{equation}\label{eq:LMOpr}
    \min_{\by\in \bar\Delta^i_{K-1}} \bg_i^\top \by,
    \end{equation}
   with $\bar\Delta^i_{K-1}:=\{\by \in \R^K_+ :\mathbf{1}^\top \by=1, \by_j=0\ \forall\ j\notin supp(\bu_i)\}$, while we do nothing and choose $\bu_i$ in case ${\bf u}_i \in\{0,1\}^K$.
		It is hence easy to see that the \GLMO\ cost   is $\mathcal{O}(|{\cal I}|\cdot m)$, with $m=\max_{i\in {\cal I}} |supp(\bu_i)| $.

The overall algorithm is given in \Cref{alg:FW2} and we now report a detailed analysis of its convergence properties.

    \begin{theorem}[Convergence Rate GFW - General Case]\label{th:nonconvb}
Consider problem \eqref{penalty_form} with $\varepsilon< \tilde \varepsilon=1/\left[ K(\rho^{max}+d^{max}_{\omega})\right]$, where $\rho^{max}=\max_{i}\sum_{j=1}^n|(\bL_s)_{ij}|$.    
Let $\{\bU_k\}$ be a sequence generated by the GFW algorithm, where the step size $\alpha_k$ satisfies
 \begin{equation}    \label{alphabound}
	{{\alpha}_k  \ge \bar{\alpha}_k = \min \left(1, \frac{g_k}{ L\n{\bD_k}_F^2} \right)   \, ,}
\end{equation}
as well as
\begin{equation} \label{eq:rho}
	E(\bU_k) - E(\bU_k+ \alpha_k\bD_k) \geq \rho\bar{\alpha}_k g_k  \quad \textnormal{for some fixed } \rho > 0 \, ,
	\end{equation}
with $g_k = -\ps{\nabla E(\bU_k)}{\bD_k}$.
	Then for every $T \in \mathbb{N}$
	\begin{equation} \label{g_T}
			g_T^* \leq \max\left(\sqrt{\frac{L 2n(E(\bU_0) - E^*)}{\rho T}}, \frac{2(E(\bU_0) - E^*)}{T} \right)\,  ,
	\end{equation}
where $\displaystyle g^*_k = \min_{0 \leq i \leq k-1} g_i$ 
 and $E^* = \min_{\bU \in \Sigma} E(\bU)$.
\end{theorem}

It is crucial to observe that the conditions specified in equations~\eqref{alphabound}-\eqref{eq:rho} can be satisfied by employing appropriate line search techniques. For further elucidation, one may refer to the works of Bomze et al.~\cite{bomze2020active}, \cite{bomze2021frank}, and Rinaldi et al.~\cite{rinaldi2022avoiding}. Notably, \Cref{nonstationary} presented herein demonstrates that these conditions are particularly met when using the Armijo line search method, which determines
\begin{equation}\label{eta}
\alpha_k = \delta^j,
\end{equation}
where $j$ is the smallest non-negative integer such that
\begin{equation}
\label{Armijo}
E(\bU_k) - E(\bU_k + \alpha_k \bD_k) \geq -\gamma \alpha_k \ps{\nabla E(\bU_k)}{\bD_k} ,
\end{equation}
with $\gamma\in (0,1/2)$ and $\delta \in (0,1)$ being two fixed parameters.

\begin{lemma} 
\label{nonstationary}
 If $\alpha_k$ is determined by the Armijo line search described above then 
	\begin{equation} \label{Arbaralpha}
		\alpha_k \geq \min \left(1,2\delta(1 - \gamma) \frac{g_k}{L\n{\bD_k}_F^2} \right) \geq \min\{1, 2\delta(1 - \gamma)\} \bar{\alpha}_k \, ,
	\end{equation}	
	with $\bar \alpha_k =  \min \left(1,  \frac{g_k}{L\n{\bD_k}_F^2} \right)$ as in~\cref{alphabound}, and \cref{eq:rho} holds with $\rho = \gamma\min\{1,2\delta(1- \gamma)\}<1.$ 
\end{lemma} 

In Lemma \ref{nonstationary} we prove that 
	$$\alpha_k \geq \min \left(1, c \frac{g_k}{L\n{\bD_k}_F^2} \right)  \textnormal{ for some } c > 0 \, , $$ 
	for the Armijo line search.
When $c \geq 1$ then \cref{alphabound} is of course a lower bound for the step size $\alpha_k$, and when $c < 1$ we can still recover \cref{alphabound} by considering $\tilde{L} = \frac{L}{c}$ instead of $L$ as Lipschitz constant.\\

Now we show that when the parameter $\varepsilon$ is sufficiently small, the Greedy Frank-Wolfe algorithm gives a stationary point in one step, thus being equivalent to  the  simple algorithmic scheme reported below in Algorithm \ref{alg:FW3}.

\begin{algorithm}[H]
			\caption{\texttt{One-Shot Frank-Wolfe (OSFW) method}}
			\label{alg:FW3}
			\begin{algorithmic}
				\par\vspace*{0.1cm}
				\item$\,\,\,1$\hspace*{0.1truecm} Choose a point $\bU_0\in \Sigma$
				
				\item$\,\,\,2$\hspace*{0.1truecm} Compute $\bS_0 \in \GLMO_{\Sigma}(\nabla E(\bU_0))$
				\item$\,\,\,3$\hspace*{0.1truecm} Set $\bU_1=\bS_0$
				\par\vspace*{0.1cm}
			\end{algorithmic}
\end{algorithm}

\begin{corollary}[One-Shot Frank-Wolfe Convergence]\label{nonconvb2}
Consider problem \eqref{penalty_form} with $\varepsilon< \min\{\bar \varepsilon,\tilde \varepsilon\}$, where $\bar\varepsilon$ and $\tilde \varepsilon$ are chosen as in Theorem \ref{th:equivalence} and  Theorem \ref{th:nonconvb}, respectively. Then GFW gives a stationary point in one iteration, and it is equivalent to OSFW. 
\end{corollary}
We can now provide an in-depth analysis of the computational complexity associated with the method introduced in the preceding section. In \cref{th:nonconvb}, we established that the duality gap $g_T^*$ exhibits a sublinear convergence rate, specifically $g_T^* \leq \max(c_1 T^{-\frac{1}{2}}, c_2 T^{-1})$, where $c_1$ and $c_2$ are suitable constants. From this, complexity outcomes can be directly derived using conventional principles from information-based complexity theory, as detailed in \cite{nesterov2003introductory}. Specifically, in the context of our study, we demonstrate a worst-case complexity of $\mathcal{O}(\epsilon^{-2})$ concerning the iteration count required to reduce $g_T^*$ below the threshold $\epsilon$. Furthermore, an upper limit on the arithmetic operations performed in each iteration of the GFW algorithm can be easily determined: each iteration requires one gradient evaluation, having a cost $\mathcal{O}(n)$ (this usually depends on the application),   and the \GLMO\ cost  (with the line search having a computational cost of $\mathcal{O}(1)$, provided the Lipschitz constant $L$ is known).  In summary, each iteration of the GFW algorithm presents a computational burden of $\mathcal{O}(n)$, implying that $\mathcal{O}(n\epsilon^{-2})$ arithmetic operations are necessary to achieve $g_T^*$ below $\epsilon$.

We finally note that the proposed \texttt{GFW}\ algorithm
naturally supports efficient inference on previously unseen data points. When new samples are added to the dataset (i.e., nodes are added to the graph), the way the \GLMO\ is defined indeed ensures that, by suitably warm-starting the \texttt{GFW} with a point embedding the binary solution related to the original data, only the rows corresponding to the new samples are updated in the algorithm iterates, while the rows related to the old binary solution remain unchanged. Moreover, computing these updates requires just the gradient rows associated with the new nodes, which in turns depend solely on their connections with the other graph nodes. As a result, the new labels can be assigned by simply running a few \texttt{GFW}\ iterations. Since each iteration involves  sparse matrix–vector products of size determined by the new nodes, this inference step might be even cheaper than the original training phase, thus making the procedure highly scalable.

\section{Numerical experiments} \label{sec:numexp}
In this section, we illustrate the performance of the new Greedy Frank-Wolfe (GFW) method when solving multiclass segmentation formulated using the new energy function $E$ in (\ref{gradgl_vv2}). In particular, we compare GFW against the classic PDE graph-based Convexity Splitting (CS) method and the Merriman-Bence-Osher (MBO) method on segmentation problems both for networks and images (color segmentation). After introducing some implementation details and our test set, we first perform a preliminary parameter analysis for the considered models (cf. the objectives (\ref{gl_vv}) and (\ref{gl_vv2})), and then compare the three algorithms on different data sets.

\subsection{Implementation issues and test problems}
The CS, MBO and GFW methods have been implemented in MATLAB\footnote{MATLAB R2019b on a Intel Core i7-9700K CPU @ 3.60GHz x 8, 16 GB RAM, 64-bit.} and are available at \url{https://github.com/mstoll1602/GreedyFrankWolfe}.
Regarding the details of the implementation of CS and MBO we refer the reader to \cite{BerF12,MerKB13}.  The low-rank eigenvalue approximations for both methods are computed with built-in MATLAB functions.

The Greedy Frank Wolfe method has been implemented following  \Cref{alg:FW2}. The method is stopped when either $30$ iterations are performed or when
$g_k = -\ps{\nabla E(\bU_k)}{\bD_k} \le 10^{-6} $. Moreover, we implemented a classical Armijo line search strategy to compute the step size $\alpha_k$ in Step 6 of  \Cref{alg:FW2} with $\delta = 0.5 $ and $\gamma= 10^{-6}$ in (\ref{eta}) and (\ref{Armijo}), respectively.

The initial guess $\bU_0$ for the algorithms CS and MBO is produced in accordance with the procedures outlined in \cite{BerF12,MerKB13}. More specifically, we begin by drawing random numbers from a standard uniform distribution defined over the interval $(0,1)$. Subsequently, these randomly generated values are projected row-by-row onto the unit simplex $\Delta^i_{K-1}$. Furthermore, at locations corresponding to fidelity points, these values are adjusted to correspond to pure phases.
On the other hand, the initialization $\bU_0 = \hat \bU$ is used for the FW algorithms.


We have tested our algorithms on both networks for data segmentation and images for image labeling, as detailed in Tables \ref{tab:names22} and \ref{tab:names_im}, respectively. Tables report the number $n$ of nodes (i.e. pixels for images) in the data set, the number of classes $K$ and the sparsity sp($\bL_s$) of the Laplacian matrix $\bL_s$ (only for the networks in Table \ref{tab:names22}). 

In detail, we considered the widely used Lancichinetti-Fortunato-Radicchi (LFR) benchmark data set \footnote{LFR networks are available at \url{https://github.com/GiulioRossetti/cdlib\_datasets}. The graph parameters reported in Table \ref{tab:names22} correspond to LFR(number of nodes, average degree, min community, mixing coefficient).} for different numbers of nodes and mixing coefficients (0.1 and 0.2) and the 5 real networks: Twitch,  LastFM, Facebook and Amazon (computers and photos). 
In all the networks, for each class we correctly labeled 1/3 of randomly sampled nodes, 1/10 for LastFM (the unknown nodes are initialized as $1/K$ for each class). For all data sets,  the matrix $\bL_s$ is retrieved from the available graph adjacency matrix and results in a sparse matrix. Regarding the imaging problems in Table \ref{tab:names_im}\footnote{The {\em beach} image can be downloaded from \url{https://github.com/mstoll1602/GreedyFrankWolfe}, while the other three images are available at \url{https://github.com/Visillect/colorsegdataset}},  the matrix-vector products with the graph Laplacian is performed using an NFFT approach \cite{alfke2018nfft}. This function approximates the matrix vector product with the fully connected graph Laplacian while avoiding the complexity of having to deal with a dense matrix and the corresponding numerical cost. The images are resized versions of the original ones. For the {\em beach}, initial and labeled images are shown in the forthcoming Figure \ref{fig:beach}, for the others in the Supplementary Materials (SM), Figures SM2, SM4 and SM6.

In the next sections, solvers will be compared mainly using two performance measures: the accuracy ($\mathcal{A}$), computed as the percentage of correctly classified nodes (correctly classified pixels in the case of images), and the elapsed time ($\mathcal{T}$) in seconds. 
In order to compute $\mathcal{A}$ for image labeling, we generated their ground truths as explained in SM Section SM1.1.

\begin{table}[!ht]
\small
    \centering
    \begin{tabular}{l r r r | l r r r }
    \toprule
   Data set & $n$ & $K$ & sp($\bL_s$) & Data set  & $n$ & $K$ & sp($\bL_s$)  \\
    \midrule    
     LFR (1000, 5, 50, 0.1/0.2)  \cite{LFR}& 1000  & 11 & $5\,10^{-3}$   & Twitch \cite{snapnets,rozemberczki2021twitch}& 168114  & 21 & $4\,10^{-4}$ \\ 
  LFR (5000, 5, 50, 0.1/0.2)  \cite{LFR}& 5000  & 62 & $10^{-3}$  & LastFM \cite{snapnets,feather}& 7624  & 9 & $10^{-3}$\\ 
    LFR (10000, 5, 50, 0.1/0.2) \cite{LFR}& 10000  & 120 & $5\,10^{-4}$ & Facebook \cite{snapnets,rozemberczki2019multiscale}& 22470  & 4 & $7\,10^{-4}$\\ 
  LFR (50000, 5, 50, 0.1/0.2) \cite{LFR}& 50000  & 426 & $10^{-4}$  & Amazon (pc) \cite{amazon} & 13752  & 10  & $2\,10^{-3}$\\ 
   LFR (100000, 5, 50, 0.1/0.2) \cite{LFR}& 100000  & 585 & $5\,10^{-5}$ & Amazon (pic) \cite{amazon} & 7650  & 8  & $4\,10^{-3}$\\ 
          \bottomrule
   \end{tabular}
    \caption{Data set details for data segmentation.}
    \label{tab:names22}
\end{table}

\begin{table}[!ht]
\small
    \centering
    \begin{tabular}{ l r r  |  l r r }
    \toprule
    Image & $n$ & $K$ & Image & $n$ & $K$ \\
    \midrule
   {\em Beach} \cite{bergermann2021semi} & 38220  & 4 &  {\em 4 geometric figures} \cite{smagina2019linear}& 25920  & 5 \\ 
    {\em 3 geometric figures} \cite{smagina2019linear}& 35828  & 4  & {\em Sheets of paper} \cite{barnard2002data,smagina2019linear}& 16500  & 8 \\ 
             \bottomrule
   \end{tabular}
    \caption{Data set details for image labeling.}
    \label{tab:names_im}
\end{table}


\subsection{Preliminary parameter tuning}

In this section, we investigate the role of the weights $\omega$, $\mu$ and $\varepsilon$ within the objective functions (\ref{gl_vv}) and (\ref{gl_vv2}) and in (\ref{mbo_update}), which correspond to different graph-based semi-supervised learning models for multiclass segmentation.

Earlier studies \cite{bergermann2021semi,BerF12,GarMBFP14} provide limited discussion regarding how the CS and MBO methods rely on the parameters $\omega_0$ and $\mu$. Therefore, we considered the values $\omega_0 = 10^i$, $i \in \{1, 2, 3, 4\}$ and $\mu \in \{5\,10^2, 10^2, 5\, 10, 10, 5\, 10^{-1}, 10^{-1}, 5\,10^{-2}, 10^{-2}, \- 5\, 10^{-3}\}$  (for CS only) and noticed a variation from 10\% to 20\% in terms of accuracy in the computed  solution. Therefore, in the experiments, we used the parameters that gave the best accuracy on an average of 10 runs. These best parameters are reported in Table \ref{tab:bestCSMBO}.
The table also reports the number $K_{eig}$ of the eigenvalues/eigenvectors chosen for each test problem. We note that the accuracy obtained using CS and MBO strongly depends on this choice and that
we selected the smallest value that gives the highest accuracy in our experience. Clearly, a large value of $K_{eig}$ (e.g., for Facebook and Amazon networks) may also impact on the overall elapsed time.

\begin{table}[!ht]
\small
    \centering
    \begin{tabular}{l r r  r r r r  }
    \toprule
    &  \multicolumn{1}{c}{LFR ($n\le 10000) $ } & \multicolumn{1}{c}{LFR ($n\ge 50000) $ } &  \multicolumn{1}{c}{Twitch} &  \multicolumn{1}{c}{LastFM} &  \multicolumn{1}{c}{Facebook} &  \multicolumn{1}{c}{Amazon} \\
    \midrule
  $\omega_0$ & 100 & 100 & 100 & 1000 & 100  & 100  \\ 
  $\mu $     & $100$ & $100$ & $10$ & 50 & 10 & 500  \\
  $K_{eig}$  & 100  & 200 & 6  & 15  & 100 & 150 \\ 
        \bottomrule
   \end{tabular}
\vskip 0.2cm
      \begin{tabular}{l c r r r r  }
    \toprule
    &  \multicolumn{1}{c}{\em Beach} &  \multicolumn{1}{c}{\em 3 geometric figures} &  \multicolumn{1}{c}{{\em 4 geometric figures}} &  \multicolumn{1}{c}{{\em Sheets of paper}} &  \multicolumn{1}{c}{ } \\
    \midrule
  $\omega_0$ & $10^{5}$ & $10^{5}$& $10^{5}$ &$10^{5}$  &   \\ 
  $\mu $     & $10^{-1}$ & $10^{-1}$ & $10^{-1}$ & $10^{-1}$ &   \\
  $K_{eig}$  & 7 (CS) and 3 (MBO)  & 6   & 6  & 20 & \\
        \bottomrule
   \end{tabular}
    \caption{Best parameter configuration for CS and MBO.}
    \label{tab:bestCSMBO}
\end{table}

We also considered values $\omega_0 = 10^i$, $i \in \{1, 2, 3, 4\}$ and $\varepsilon \in \{5\,10^2, 10^2, 5\, 10, 10,  5\, 10^{-1}, 10^{-1}, \\5\,10^{-2}, 10^{-2}, 5\, 10^{-3}\}$ in the objective  (\ref{gl_vv2}) and analyzed the performance of GFW in terms of both accuracy and elapsed time.
Figure \ref{fig:bestGFW} shows these values varying $\varepsilon$ and $\omega_0$ for the Facebook network as a representative example for all the datasets. The figure shows a double behaviour of GFW: for some choices of the parameters (those corresponding to blue dots) GFW terminates in one step, being in fact OSFW and therefore the fastest version; for some other values of the parameters (those corresponding to red dots) GFW reaches the maximum accuracy.
Therefore, in the next section we use  the following parameter values
$$\omega_0 = 1000
\quad \mbox{ and } \quad \varepsilon = 50,$$ for experiments on data from networks. For image labeling we used $$\omega_0=10^5 \quad \mbox{ and } \quad \varepsilon = 10^{-1}$$
that guarantee the one-shot behavior for GFW, see \Cref{alg:FW3} and \Cref{nonconvb2}. 

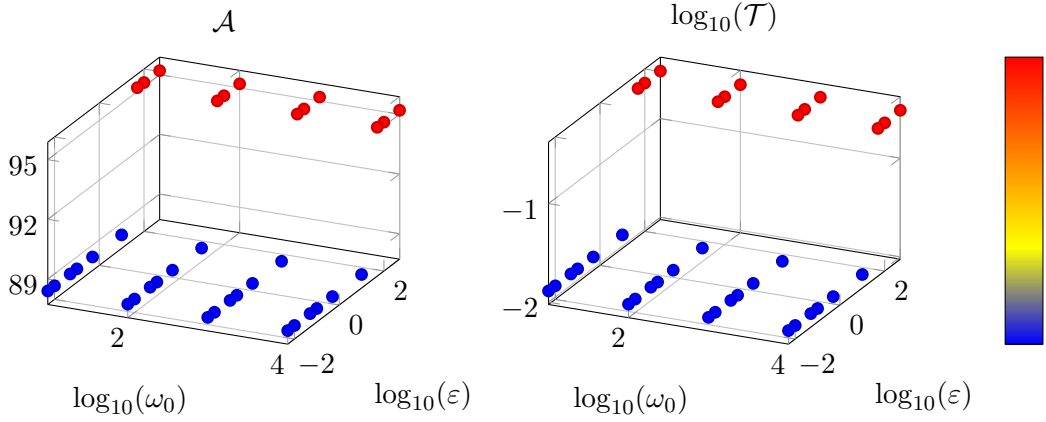
\begin{figure}
\begin{tikzpicture}
\begin{axis}[width=.4 * \linewidth, 
    xlabel=$\log_{10}(\omega_0)$,grid=major,ztick={89,92,95},	
    ylabel=$\log_{10}(\varepsilon)$,
    title=$\mathcal{A}$,
]
    \addplot3[thick, only marks, scatter]  table {
        x   y   z
       1.0000    2.6990   95.2000
    1.0000    2.0000   95.2000
    1.0000    1.6990   95.2000
    1.0000    1.0000   88.4000
    1.0000   -0.3010   88.4000
    1.0000   -1.0000   88.4000
    1.0000   -1.3010   88.4000
    1.0000   -2.0000   88.4000
    1.0000   -2.3010   88.4000
    2.0000    2.6990   95.2000
    2.0000    2.0000   95.2000
    2.0000    1.6990   95.2000
    2.0000    1.0000   88.4000
    2.0000   -0.3010   88.4000
    2.0000   -1.0000   88.4000
    2.0000   -1.3010   88.4000
    2.0000   -2.0000   88.4000
    2.0000   -2.3010   88.4000
    3.0000    2.6990   95.2000
    3.0000    2.0000   95.2000
    3.0000    1.6990   95.2000
    3.0000    1.0000   88.4000
    3.0000   -0.3010   88.4000
    3.0000   -1.0000   88.4000
    3.0000   -1.3010   88.4000
    3.0000   -2.0000   88.4000
    3.0000   -2.3010   88.4000
    4.0000    2.6990   95.2000
    4.0000    2.0000   95.2000
    4.0000    1.6990   95.2000
    4.0000    1.0000   88.4000
    4.0000   -0.3010   88.4000
    4.0000   -1.0000   88.4000
    4.0000   -1.3010   88.4000
    4.0000   -2.0000   88.4000
    4.0000   -2.3010   88.4000
    };
\end{axis}
\end{tikzpicture}
\begin{tikzpicture}
\begin{axis}[colorbar, colorbar style={ticks = none},width=.4 * \linewidth, 
    xlabel=$\log_{10}(\omega_0)$, ztick={-2,-1,0},	
    grid=major,
    ylabel=$\log_{10}(\varepsilon)$,
    title=  $\log_{10}(\mathcal{T})$,]
    \addplot3[thick, only marks, scatter]  table {        x   y   z
    1.0000    2.6990   -0.5186
    1.0000    2.0000   -0.5200
    1.0000    1.6990   -0.5229
    1.0000    1.0000   -1.8861
    1.0000   -0.3010   -1.8861
    1.0000   -1.0000   -1.8861
    1.0000   -1.3010   -1.8861
    1.0000   -2.0000   -1.8861
    1.0000   -2.3010   -1.8861
    2.0000    2.6990   -0.5229
    2.0000    2.0000   -0.5243
    2.0000    1.6990   -0.5229
    2.0000    1.0000   -1.8861
    2.0000   -0.3010   -1.8861
    2.0000   -1.0000   -1.8861
    2.0000   -1.3010   -1.8861
    2.0000   -2.0000   -1.8861
    2.0000   -2.3010   -1.8861
    3.0000    2.6990   -0.5157
    3.0000    2.0000   -0.5186
    3.0000    1.6990   -0.5258
    3.0000    1.0000   -1.8861
    3.0000   -0.3010   -1.8861
    3.0000   -1.0000   -1.8861
    3.0000   -1.3010   -1.8861
    3.0000   -2.0000   -1.8861
    3.0000   -2.3010   -1.8861
    4.0000    2.6990   -0.5143
    4.0000    2.0000   -0.5214
    4.0000    1.6990   -0.5258
    4.0000    1.0000   -1.8861
    4.0000   -0.3010   -1.8861
    4.0000   -1.0000   -1.8861
    4.0000   -1.3010   -1.8861
    4.0000   -2.0000   -1.8861
    4.0000   -2.3010   -1.8861
    };
    \end{axis}
\end{tikzpicture}
\caption{Facebook network: values of accuracy ($\mathcal{A}$) and base-10 logarithm of elapsed time ($\log_{10}(\mathcal{T})$) varying the parameters $\omega_0$ and $\varepsilon$ in the objective function (\ref{gradgl_vv2}).}\label{fig:bestGFW}
\end{figure}


\subsection{Numerical results}
Table \ref{tab:res_net} collects averaged results obtained using GFW 
and the competing solvers CS and MBO on network data (Table \ref{tab:names22}), over 10 runs. The symbol $\infty$ means that the elapsed time is above 1 hour.

Focusing on LFR problems, we observe that GFW always gives the best accuracy in short time and the gain increases as the number of nodes increases. Moreover, while the elapsed time for the FW solver stays bounded as $n$ and $K$ increase, both CS and MBO take longer time.
In details, these runs highlight on the main computational bottleneck of CS and MBO, that is
the cost of the projections over simplices that is quite high when the number of vertices is large.
Moreover, the large elapsed time of CS is due to the evaluation of the elements $T_{i,j}$
in (\ref{eq_Tik}) that is rather time consuming as $K$ increases. 

Focusing on the real networks, GFW always roughly computes the solution with 10\% of better accuracy wrt CS and MBO in much less time, except for the Twitch network for which MBO is the fastest. 
In fact, for Twitch MBO takes a short time as it employs only few iterations (and then only few projections). 
Finally, we remark that for Facebook and Amazon the elapsed time of CS and MBO is affected by the nonnegligible cost of the initial eigendecomposition of, although sparse, large matrices (see Table \ref{tab:bestCSMBO}).

\begin{table}[!ht]
    \centering
    \begin{tabular}{l| r r | r r| r r}
    \toprule
  \multicolumn{1}{c}{}  &  \multicolumn{2}{c}{GFW  } &
\multicolumn{2}{c}{CS}   &  \multicolumn{2}{c}{MBO} \\ \midrule 
Data set  & $\mathcal{T}$ & $\mathcal{A}$ & $\mathcal{T}$ & $\mathcal{A}$& $\mathcal{T}$ & $\mathcal{A}$ \\
         \midrule
         LFR 1000/0.1&  0.02 & 99.0\% & 0.32 & 98.8\% & 0.39 & 98.4\%\\
         LFR 5000/0.1 &  0.27 & 99.2\% & 85.45 & 98.9\% & 5.45 & 98.9\%\\ 
         LFR 10000/0.1 &  0.70 & 99.4\% & 587.78 &  91.3\% & 17.31 & 96.8\% \\ 
         LFR 50000/0.1 &   17.23 & 98.1\%  & $\infty$ &  - & 2575.31& 84.8\% \\  
        LFR 100000/0.1 &  62.02 & 83.7\% & $\infty$ & - & $\infty$ & - \\ 
         \midrule
         LFR 1000/0.2   & 0.03 & 92.5\% & 0.39 & 91.1\% & 0.93 & 91.2\% \\
        LFR 5000/0.2 & 0.33  & 93.8\% & 65.5 & 92.4\% & 18.73 & 92.5\%\\
         LFR 10000/0.2 &  0.68 & 94.9\% & 807.03 & 81.5\%  & 40.99 & 86.6\%\\ 
         LFR 50000/0.2 &  18.25 & 90.9\% & $\infty$ & - & 2749.85 & 76.5\% \\ 
         LFR 100000/0.2 & 64.05 & 77.5\% & $\infty$ & - & $\infty$ & - \\ 
         \midrule
         Twitch &    15.12 & 94.9\%& 41.68& 84.6\% & 8.92 & 84.7\% \\
         LastFM &    0.12 & 85.6\% & 1.42 & 76.4\% & 1.71 & 75.9\%\\
         Facebook &   0.30 & 95.2\% & 19.84 & 74.7\% & 20.99 & 75.7\%\\
         Amazon (computer) & 0.31  & 91.8\% & 40.3  & 77.7\% & 51.7 & 81.1\%\\
         Amazon (photo) & 0.12   & 94.0\% & 25.6 & 85.9\% & 30.7  & 89.2\%\\ 
         \bottomrule
    \end{tabular}
    \caption{Network data sets of Table \ref{tab:names22}: results in terms of accuracy $\mathcal{A}$ and
    elapsed time $\mathcal{T}$ in seconds for GFW, CS and MBO.}
    \label{tab:res_net}
\end{table}

\begin{table}[!ht]
    \centering
    \begin{tabular}{l|r r|r r| r r }
    \toprule
   \multicolumn{1}{c}{}  &  \multicolumn{2}{c}{GFW} & \multicolumn{2}{c}{CS} & \multicolumn{2}{c}{MBO} \\
    \midrule
Image & $\mathcal{T}$ & $\mathcal{A}$  & $\mathcal{T}$&  $\mathcal{A}$ & $\mathcal{T}$ &  $\mathcal{A}$ \\
  \midrule
{\em Beach}  & 37.10 & 95.7\% & 87.52  & 91.2\%  & 60.96 & 84.9\%\\
  {\em 3 geometric figures}  & 57.95 & 98.6\% & 133.41 & 97.7\% & 114.22 &  87.2\%\\
 {\em 4 geometric figures}  & 35.70 & 98.4\% & 98.29 &  93.4\% & 78.92 &  83.3\%\\
 {\em Sheets of paper}  & 10.47 & 97.9\%& 45.67 &  83.7\% & 26.72 &  93.2\%\\
         \bottomrule
    \end{tabular}
    \caption{Image labeling: elapsed cpu time $\mathcal{T}$ 
    and accuracy $\mathcal{A}$.   }
    \label{tab:im_Res}
\end{table}

Table \ref{tab:im_Res} shows the elapsed time employed by the GFW (with one-shot behavior), CS and MBO and the accuracy obtained in the segmentation of the images in Table \ref{tab:names_im}. Evidently, the FW method is roughly faster by a factor of 2 (or 3) with respect to CS and MBO.
In these runs, CS and MBO  greatly suffer from the time-consuming initial eigendecomposition, while the computational effort GFW is mainly due to the matrix-matrix product with $\bL_s$. As for the accuracy,  GFW  returns the highest values in all experiments. 
In the case of the {\em beach} image, Figure \ref{fig:beach}, the four classes are reasonably identified by GFW, while CS predicts as "sea" some "beach" pixels and MBO completely misses the "sea", which is classified as "sky", see the confusion matrices in Figure \ref{fig:beach_cm}. 

Figures and confusion matrices related to experiments on   {\em 3} and {\em 4 geometric figures} and {\em Sheets of paper} are contained in Section SM1.2 of Supplementary Materials. 
For {\em 3 geometric figures} image, the accuracy reported in Table \ref{tab:im_Res} of MBO is misleading. Indeed, this image presents a strongly imbalanced class (larger of a factor 6, Figure SM2), consequently, even if MBO completely fails to identify another class, it returns a good value of accuracy, Figure SM3. Also for {\em 4 geometric figures} image, the classes are imbalanced, Figure SM4, and, in this case CS is unsuccessful in the classification of a class, while MBO completely miscarry two minority classes, Figure SM5. In the case of {\em sheets of paper} image, CS and MBO mixed some classes more, Figures SM6 and SM7. Therefore, even if the accuracy values seem high for CS and MBO, the quality of the segmentation is poor. Conversely, our new approach  provides  higher accuracy values and better segmentations.


\begin{figure}[htb!]
\begin{center}
 \setlength\figureheight{0.5\linewidth} 	
	\subfloat[Original (top) and labeled (bottom) image]{
 		\includegraphics[width=0.5\textwidth]{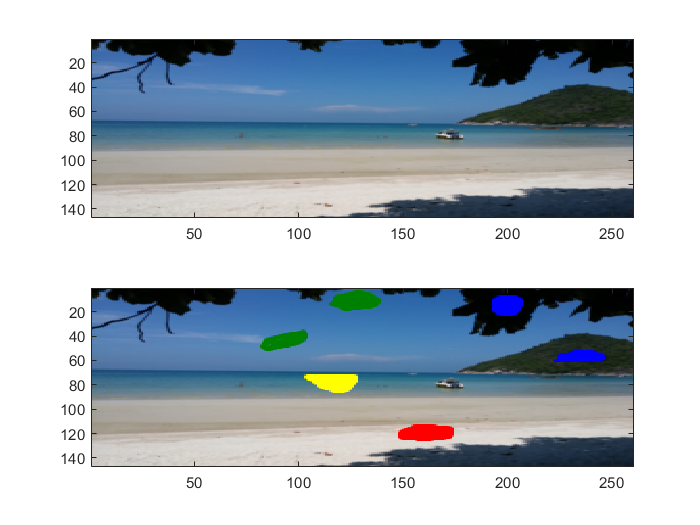}
	}
	\subfloat[Segmentation with GFW (OSFW)]{
		\includegraphics[width=0.5\textwidth]{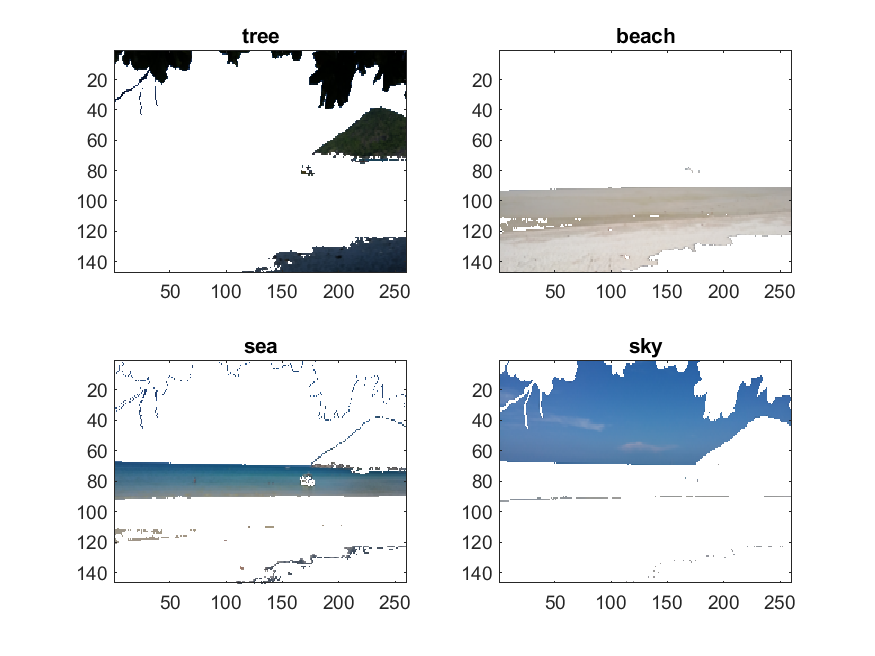}
	}
\end{center}
\begin{center}
 \setlength\figureheight{0.33\linewidth} 	
\subfloat[Segmentation with CS]{
		\includegraphics[width=0.5\textwidth]{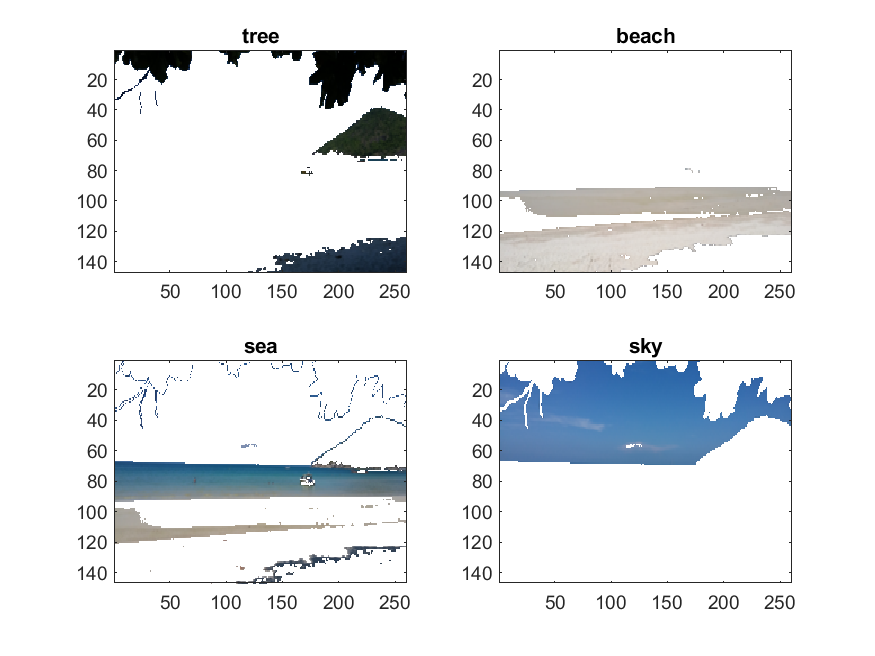}
	}
	\subfloat[Segmentation with MBO]{
		\includegraphics[width=0.5\textwidth]{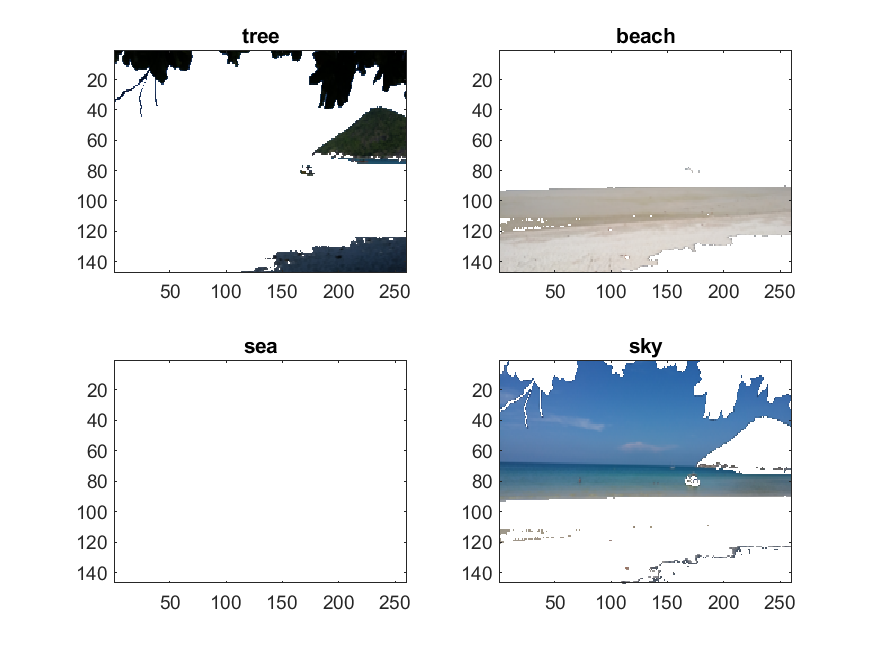}
	}
\end{center}
		\caption{{\em Beach} image: labeling results with GFW (b), CS (c) and MBO (d).}\label{fig:beach}
\end{figure}

\begin{figure}[htb!]
\begin{center}
 \setlength\figureheight{0.5\linewidth} 	
	\subfloat[Confusion matrix GFW]{
 		\includegraphics[width=0.5\textwidth]{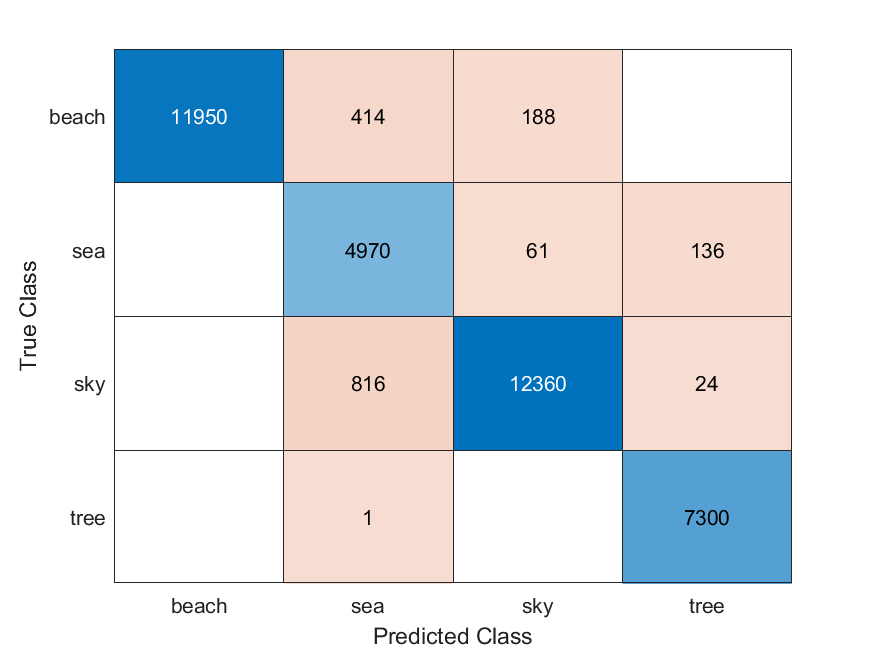}
	}
	\subfloat[Confusion matrix CS]{
		\includegraphics[width=0.5\textwidth]{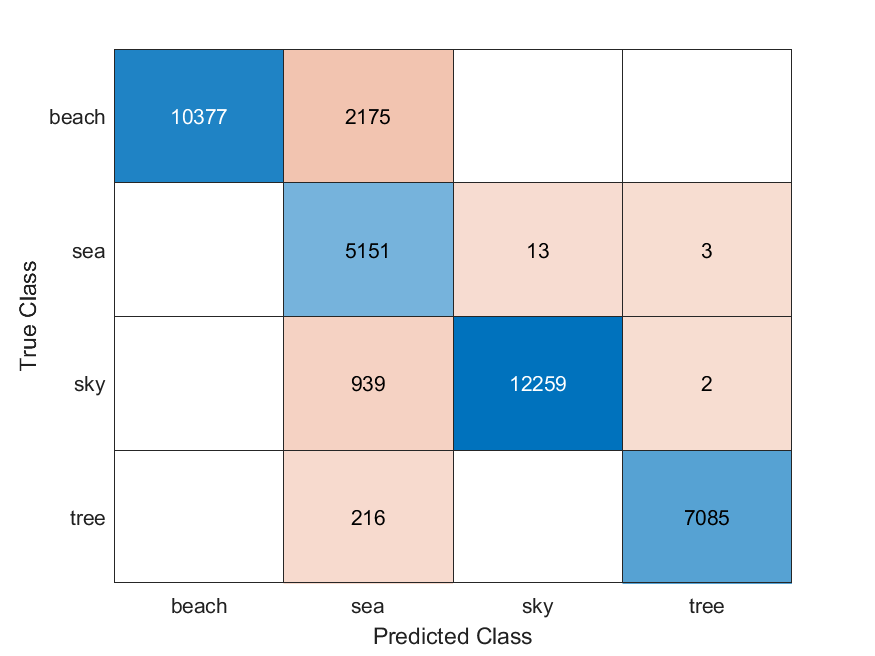}
	}
    \end{center}
\begin{center}
 \setlength\figureheight{0.33\linewidth} 	
\subfloat[Confusion matrix MBO]{
		\includegraphics[width=0.5\textwidth]{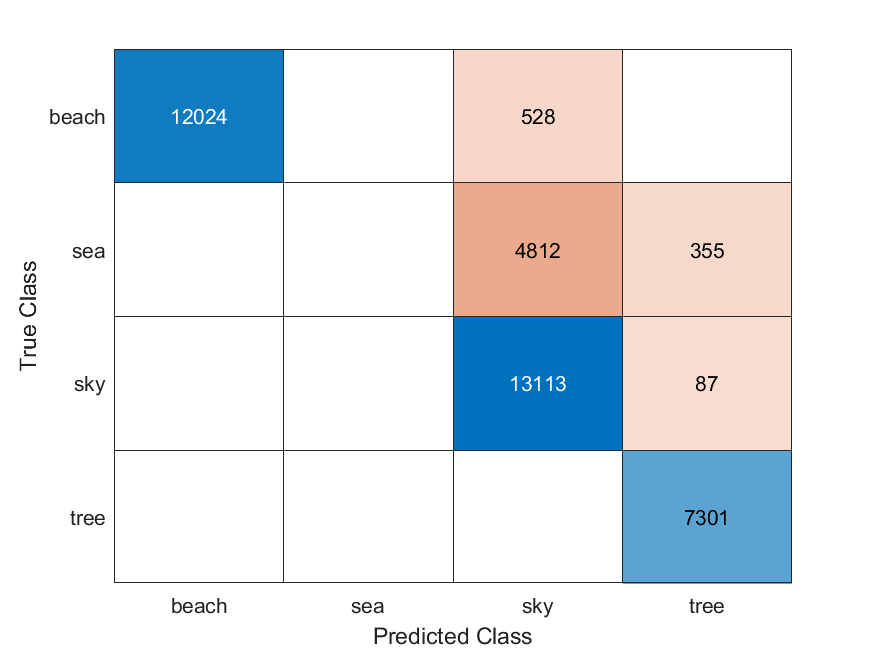}
	}
\end{center}
		\caption{{\em Beach} image: confusion matrix results with GFW (a), CS (b) and MBO (c).}\label{fig:beach_cm}
\end{figure}

\section{Conclusions and Future Work}\label{sec:conclusions}

We considered the graph-based segmentation framework and addressed key computational bottlenecks shared by classic methods in this context: the eigendecomposition of the graph Laplacian and the projection on the unit simplex. We therefore proposed a new model problem that still uses the Ginzburg-Landau energy functional but which ensures valid partitions (i.e., binary solutions) by employing a suitable penalty term. For the minimization of such a model, we defined 
a tailored Frank-Wolfe algorithm that achieves both theoretical and practical improvements  compared to traditional PDE-inspired methods (CS and MBO).  We illustrated the effectiveness of our proposal on several test problems, including real networks and images, showing both the validity of our model and the gain in terms of computational time attained with our algorithm.

The approach that has been proposed  paves the way for novel and captivating avenues of research. First, an in-depth theoretical and computational analysis related to the penalty parameter for different graph topologies might be carried out, thus strengthening the model's guarantees and giving the possibility to broaden the applicability of the method to dynamic or evolving graphs. Second, the use of acceleration (e.g., momentum-based), decomposition (e.g., block-coordinate Frank-Wolfe-like methods) and stochastic-approximation strategies (e.g., stochastic Frank-Wolfe-like methods) might help to further reduce the computational burden when dealing with massive datasets. Third, extensions that allow one to handle higher-order graph data structures (e.g., multilayer, hyperedges, and so on) might represent another interesting research topic. Finally, exploring distributed implementations could make real-time segmentation viable for large-scale imaging and network analysis.

\appendix

\section{Proof of \Cref{pr:diam}}
\begin{proof} The diameter is defined as the maximum distance between two matrices in the feasible set, that is
$$diam(\Sigma)=\max_{\bar\bU, \bU\in \Sigma}\|\bar \bU-\bU\|_F.$$
Since each row is constrained to the unit simplex, we have
$$\max_{\bar\bU, \bU\in \Sigma}\|\bar \bU-\bU\|^2_F=\sum_{i=1}^n\max_{\bar\bu_i, \bu_i\in \Sigma^{K}}\|\bar\bu_i-\bu_i\|^2_2=2n,$$
thus proving the desired result.
\end{proof}

\section{Proof of \Cref{pr:LCG}}
\begin{proof}
Let us consider the gradient of the function given in \eqref{gradgl_vv2g}. We want to prove that
$$\|\nabla E(\bar\bU)-\nabla E(\bU)\|_F\leq L\|\bar \bU-\bU\|_F.$$
We now separately bound the 3 terms of the gradient in \eqref{gradgl_vv2g}:
\begin{itemize}
\item  term $\bL_{s}\bU$. Let \( \bV = \bar\bU - \bU \), so we can write:
\[
\|\bL_s \bV\|_F^2 = \sum_{j=1}^p \|\bL_s \mathbf{v}_j\|_2^2 \leq \|\bL_s\|_2^2 \sum_{j=1}^p \|\mathbf{v}_j\|_2^2 = \|\bL_s\|_2^2 \|\bV\|_F^2.
\]

Taking square roots:
\[
\|\bL_s (\bar\bU - \bU)\|_F = \|\bL_s \bV\|_F \leq \|\bL_s\|_2 \|\bV\|_F = \|\bL_s\|_2 \|\bar\bU - \bU\|_F.
\]

We thus have for this term
$$\|\bL_{s}\bar\bU-\bL_{s}\bU\|_F=\|\bL_{s}(\bar\bU-\bU)\|_F\leq \|\bL_s\|_2\|\bar\bU-\bU\|_F\leq \lambda_{max}(\bL_s)\|\bar \bU-\bU\|_F.$$

\item  term $\frac{1}{\epsilon}(\mathbf{1} - 2 \bU)$. It is easy to see that
$$\|\frac{2}{\epsilon}(\bar\bU-\bU)\|_F\leq \frac{2}{\varepsilon}\|\bar\bU-\bU\|_F.$$
\item  term  $- \bD_\omega (\hat \bU - \bU)$.  Since  $\bD_\omega$ is diagonal, we have
$$\|\bD_{\omega}(\bU-\bar \bU)\|_F\leq \,d_{\omega}^{max}\|\bar \bU-\bU\|_F.$$
\end{itemize} 
So the result holds.
\end{proof}

\section{Proof of \Cref{th:equivalence}}
\begin{proof} If we choose a value $\varepsilon\in(0,\bar \varepsilon]$, the Hessian
$\nabla^2 E(\bU)=  \bI_K \otimes (\bL_{s}+\bD_\omega-\frac{2}{\epsilon} \bI_{n})$ is negative definite. As a result, the function $E(\bU)$ is strictly concave and every local (and global) minimum $\bU^*$ is strict and is a vertex of the feasible set $\Sigma$.
In order to see this, by contradiction assume that $\bU^*$ is not an extreme point of $\Sigma$. Then there exists a neighborhood $B(\bU^*,\rho)$, with $\rho>0$,  of suitable dimension fully contained in the feasible set. We can hence find a direction $\tilde \bD=\tilde \bU- \bU^*$, with  $\tilde\bU\in B(\bU^*,\rho)$, such that both $\tilde \bD$ and $-\tilde\bD$ are feasible directions. By first order optimality conditions we can write $$\langle\nabla E(\bU^*), \bD \rangle\geq 0,$$
for any feasible direction $\bD$. Then $\langle\nabla E(\bU^*), \tilde\bD \rangle=0$. Thanks to the way we choose $\varepsilon$, we also have
$$\langle\tilde \bD, \nabla^2 E(\bU^*) \tilde\bD \rangle<0.$$
Thus we can write 
$$E(\tilde\bU)=E(\bU^*+\tilde \bD)= E(\bU^*)+\langle\nabla E(\bU^*), \tilde \bD \rangle+\langle\tilde \bD, \nabla^2 E(\bU^*) \tilde\bD\rangle<E(\bU^*),$$
which contradicts optimality of $\bU^*$
(see, e.g., \cite{rbook} for further details). By using similar arguments, we can easily see that $\bU^*$ is strict. Since $\Sigma$ is defined as in \eqref{fs}, then  $\bU^*\in\{0,1\}^{n\times K}$ and the result (i) is proved. 

In order to prove point (ii), we  first need to show that any global minimizer of \eqref{penalty_form} is also a global minimizer of \eqref{binary_form}.
Since for any global minimizer $\bU^*$ of \eqref{penalty_form}, we have that $\bU^* \in \Sigma \cap \{0,1\}^{n\times K}=: \bar \Sigma$, and
$E(\bU^*)=\bar E(\bU^*)\leq \bar E(\bU)$
for all $\bU \in \bar \Sigma$,  then the result holds. Now we need to prove that any global minimizer of \eqref{binary_form} is also a global minimizer of\eqref{penalty_form}. By contradiction, assume there exists $\varepsilon \in (0,\bar \varepsilon]$ such that
$$E(\hat \bU)<E(\bU^*),$$
with $\hat \bU \in \Sigma$ global minimizer of \eqref{penalty_form} and $ \bU^* \in  \bar \Sigma$ global minimizer of \eqref{binary_form}.  Recalling that for all $\varepsilon \in (0,\bar \varepsilon]$  we proved that $\hat \bU \in \bar \Sigma$
we then have 
$$\bar E(\hat \bU)< \bar E(\bU^*),$$
which contradicts global optimality of $\bU^*$. Thus point (ii) is proved.
\end{proof}

\section{Proof of \Cref{th:nonconvb}}
\begin{proof}
     We first show that $\GLMO_\Sigma (\bG)=\LMO_\Sigma(\bG)$ 
     where $\bG = \nabla E(\bU)$ and $\bU\in \Sigma$, when  $\varepsilon\leq \tilde \varepsilon$. If $\bU\in \Sigma$ is such that $\bu_i\in\{0,1\}^K$ for some $i$, then it is easy to see that $(\bg_i)_j\simeq{\varepsilon}^{-1}$ if 
    $(\bu_i)_j=0$ and $(\bg_i)_j\simeq-{\varepsilon}^{-1}$ otherwise. Since in the \LMO\ we solve for each $i$ the following problem: 
    \begin{equation}\label{eq:LMOpr1}
    \min_{\by\in \Delta^i_{K-1}} \bg_i^\top \by,
    \end{equation}
    then we have $\by= \bu_i$. If $\bU\in \Sigma$ is such that $\bu_i\notin\{0,1\}^K$, but $supp(\bu_i)<K$, we have that $(\bg_i)_j\simeq{\varepsilon}^{-1}$ for all $j$ such that $(\bu_i)_j=0$. We 
    now show that thanks to the way we set $\varepsilon$ there always exists at least one component of the gradient smaller than $\varepsilon^{-1}$ and thus we have $\by_j=0$ in the solution of problem \eqref{eq:LMOpr1} when $(\bu_i)_j=0$. Notice that the gradient of the function without penalty term is such that $\|\nabla \bar E(\bU)\|_{\infty}\leq\rho^{max}+d_\omega^{max}$, furthermore the $j$-th gradient component related to the penalty is $\varepsilon^{-1}[1-2(\bu_i)_j]$, so the worst case scenario for us is when all components have the same weight, with an upper bound given by $\varepsilon^{-1}[1-2K^{-1}]$. Hence, taking into account the way we choose $\varepsilon$ in problem \eqref{penalty_form}, we have 
    $${\varepsilon}^{-1}>\varepsilon^{-1}[1-2K^{-1}]+2( \rho^{max}+d_\omega^{max})$$ and a component $(\bu_i)_j=0$ will stay zero when solving problem \eqref{eq:LMOpr1} and the result is proved.

	We now analyze the convergence rate of the method by considering two different cases. \\	
\textbf{Case 1.} $\bar{\alpha}_k < 1$. 
	It is easy to see that  $\bar{\alpha}_k =  \frac{g_k}{{L\n{\bD_k}_F^2}} $ in this case and inequality \cref{eq:rho} becomes
	\begin{equation}
    \label{c1}
	E(\bU_k) - E(\bU_k + \alpha_k\bD_k) \geq  \frac{\rho}{L{\n{\bD_k}^2_F}} g_k^2 \geq \frac{\rho g_k^2}{L2n} \, ,
	\end{equation}
    where the last inequality is obtained by considering Proposition \ref{pr:diam}. \\   
\textbf{Case 2.} $\bar{\alpha}_k = 1$. 
	We start by considering the standard descent lemma~\cite[Proposition 6.1.2]{bertsekas2015convex} applied to $E$ with center $\bU_k$ and $\alpha = 1$
	$$ E(\bU_{k+1}) = E(\bU_k + \bD_k) \leq E(\bU_{k}) + \langle\nabla E(\bU_{k}), \bD_k \rangle+ \frac{L}{2}\|\bD_k\|_F^2\, .  $$
	The condition that characterizes Case 2 gives $\min \left(\frac{-\langle\nabla E(\bU_{k}), \bD_k \rangle}{\|\bD_k\|_F^2L}, 1\right) =  \alpha_k = 1 $, thus we can write 
	\begin{equation*}
	\frac{-\langle\nabla E(\bU_{k}), \bD_k \rangle}{\|\bD_k\|_F^2L} \geq 1 \, \mbox{, and }\quad -L\n{\bD_k}_F^2 \geq  \langle\nabla E(\bU_{k}), \bD_k \rangle\, ,
	\end{equation*} 
	hence we obtain
	\begin{equation} \label{c2}
	E(\bU_{k}) - E(\bU_{k+1}) \geq -\langle\nabla E(\bU_{k}), \bD_k \rangle - \frac{L}{2}\|\bD_k\|_F^2 \geq - \frac{\langle\nabla E(\bU_{k}), \bD_k \rangle}{2} \geq \frac{g_k}{2}\, . 
	\end{equation}\\

Now, taking into account the two cases analyzed above, we can partition the iterations $\{0,1,\ldots,T-1\}$ into two subsets $N_1$ and $N_2$ defined as follows:
\[
N_1 = \{k < T \colon \bar{\alpha}_k < 1\}, \quad
N_2 = \{k < T \colon \bar{\alpha}_k = 1\} \, .
\]

Using~\eqref{c1} and~\eqref{c2}, we have:
        \begin{equation}
        \begin{split}
        E(\bU_{0})-E^*  &\ge \sum_{k=0}^{T-1} (E(\bU_{k})-E(\bU_{k+1})) \\
        & = \sum_{k \in N_1} (E(\bU_{k})-E(\bU_{k+1}))+
        \sum_{k \in N_2} (E(\bU_{k})-E(\bU_{k+1})) \\
        & \ge \sum_{k \in N_1} \frac{\rho g_k^2}{L2n} +
        \sum_{k \in N_2} \frac{g_k}{2} \\
        & \ge |N_1| \min_{k \in N_1} \frac{\rho g_k^2}{L2n} + |N_2| \min_{k \in N_2} \frac{g_k}{2}\\
        & \ge (|N_1|+|N_2|) \min \left(\frac{\rho (g_T^{*})^2}{L2n} ,\frac{g_T^{*}}{2} \right)\\
        & = T \min \left(\frac{\rho (g_T^{*})^2}{L2n} ,\frac{g_T^{*}}{2} \right)
        \, ,
        \end{split}
        \end{equation}
        where the last inequality is obtained using the definition of $g_T^{*}$.\\
	Now, if $T \min \left( \frac{\rho(g^*_{T})^2}{L2n},\frac{g^*_{T}}{2}  \right) =\frac{Tg^*_{T}}{2} $ we then have 
	\begin{equation} \label{g*}
	g^*_{T} \leq \frac{2(E(\bU_0) - E^*)}{T} \, ,
	\end{equation}
	and otherwise
	\begin{equation} \label{g*sqrt}
	g^*_{T} \leq \sqrt{\frac{L2n(E(\bU_0) - E^*)}{\rho T}}\, .
	\end{equation}
	The result is thus proved  by just considering the max between the right-hand sides in equation~\cref{g*} and \cref{g*sqrt}.
\end{proof}

\section{Proof of \Cref{nonstationary}}
\begin{proof}
Using the standard descent lemma~\cite[Proposition 6.1.2]{bertsekas2015convex} we obtain
	\begin{equation} \label{e11:std}
	E(\bU_k) - E(\bU_k + \alpha \bD_k) \geq \alpha g_k - \alpha^2 \frac{L\|\bD_k\|_F^2}{2}\, .
	\end{equation}
It follows 
	\begin{equation}
		E(\bU_k) - E(\bU_k + \alpha \bD_k) \geq \gamma \alpha g_k \quad \textnormal{for } \alpha \in \left[0, 2(1 - \gamma) \frac{g_k}{L\n{\bD_k}_F^2} \right] \, ,
	\end{equation}
	and
	$$\alpha_k >2\delta(1 - \gamma) \frac{g_k}{L\n{\bD_k}_F^2} \, . $$
	Therefore
	\begin{equation} \label{1Ar}
		\alpha_k \geq \min \left(1,2\delta(1 - \gamma) \frac{g_k}{L\n{\bD_k}_F^2} \right) \geq \min\{1,2\delta (1- \gamma)\}\bar{\alpha}_k \, ,
	\end{equation}
	which proves \cref{Arbaralpha}. Furthermore, it holds that
	\begin{equation}
		E(\bU_k) - E(\bU_k+ \alpha_k \bD_k) \geq \gamma \alpha_k g_k \geq \gamma\min\{1,2\delta (1- \gamma)\}\bar{\alpha}_k g_k \, ,
	\end{equation}
	by using  the Armijo condition \cref{Armijo} in the first inequality and \cref{Arbaralpha} in the second. 
	Thus, by $\gamma,\delta\in(0,1)$ and $\gamma(1-\gamma)\leq \frac{1}{4}$, our analysis shows that equation \cref{eq:rho} is satisfied given that $\rho$ is defined as $\rho = \gamma \min\{1, 2\delta (1 - \gamma)\}$ and that this value remains less than 1.
\end{proof} 

\section{Proof of \Cref{nonconvb2}}
\begin{proof} As proved in Theorem \ref{th:equivalence}, when $\varepsilon< \bar \varepsilon$ the function $E(\bU)$ is strictly concave. Since the point $\bS_0$ obtained through the \GLMO\ is  a vertex of the feasible set, so we have $\bS_0\in\{0,1\}^{n\times K}$. By exploiting strict concavity, we know that the univariate function $\phi(\alpha)=E(\bU+\alpha \bD_k)$ has its minimum in $\alpha=1$, which is the stepsize that a linesearch embedded in GFW would give. By exploiting first order strict concavity we can indeed write 
$$E(\bU_0) - E(\bU_1) > -\langle \nabla E(\bU_0), (\bU_1-\bU_0) \rangle > 0 \, .$$

Therefore $ \bU_1 \in \{0,1\}^{n\times K}$, and since $\varepsilon<\tilde \varepsilon$, we have that $\GLMO(\nabla E(\bU_1))$ (and the classic \LMO) would output $\bU_1$ meaning that $\bU_1$ is a stationary point for problem \eqref{penalty_form}. This proves our result.
\end{proof}

\end{document}